\documentclass[11pt]{amsart}
\usepackage{fullpage}
\usepackage{color}
\usepackage{graphicx,psfrag} 
\usepackage{color} 
\usepackage{tikz}
\usepackage{pgffor}
\usepackage{hyperref}
\usepackage{todonotes}
\usepackage{eqnarray,amsmath}
\usepackage{perpage}     
\usepackage{wrapfig}
\usepackage{enumerate}
\usepackage{subfigure}
\usepackage{caption}
\usepackage[normalem]{ulem}
\usepackage[makeroom]{cancel} 
\usepackage{amssymb} 

\MakePerPage{footnote}   

\newtheorem{theorem}{Theorem}[section]
\newtheorem{lemma}[theorem]{Lemma}
\newtheorem{corollary}[theorem]{Corollary}
\newtheorem{proposition}[theorem]{Proposition}

\theoremstyle{definition}
\newtheorem{definition}[theorem]{Definition}
\newtheorem{example}[theorem]{Example}
\theoremstyle{remark}
\newtheorem{remark}[theorem]{Remark}

\newtheorem{ex-prop}[theorem]{Example-Proposition}

\newtheorem{algorithm}[theorem]{Algorithm}

\numberwithin{equation}{section}

\definecolor{gray}{rgb}{.5,.5,.5}

\definecolor{black}{rgb}{0,0,0}

\definecolor{blue}{rgb}{0,0,1}

\definecolor{red}{rgb}{1,0,0}

\definecolor{green}{rgb}{0,1,0}

\definecolor{yellow}{rgb}{1,1,.4}

\definecolor{purple}{rgb}{1,0,1}

\definecolor{gold}{rgb}{.5,.5,.2}

\definecolor{darkgreen}{rgb}{0,.5,0}

\definecolor{greenbean}{RGB}{199, 237, 204}

\definecolor{RED}{rgb}{1,0,0}

\newcommand{\crossone}{
\begin{tikzpicture}[baseline=-2]
\draw[white,line width=1.5pt,double=black,double distance=1.5pt] (0,-0.1) -- (0.3,0.2);
\draw[white,line width=1.5pt,double=black,double distance=1.5pt] (0,0.2) -- (0.3,-0.1);
\end{tikzpicture}}

\newcommand{\crosszero}{
\begin{tikzpicture}[baseline=-2]
\draw[white,line width=1.5pt,double=black,double distance=1.5pt] (0,0.2) -- (0.3,-0.1);
\draw[white,line width=1.5pt,double=black,double distance=1.5pt] (0,-0.1) -- (0.3,0.2);
\end{tikzpicture}}

\begin{document}
 
\title{Crossing numbers of Random Two-Bridge Knots}

\author{Moshe Cohen}
\address{Andrew and Erna Viterbi Faculty of Electrical Engineering, Technion -- Israel Institute of Technology, Haifa 32000, Israel}
\email{mcohen@tx.technion.ac.il}

\author{Chaim Even-Zohar}
\address{Einstein Institute of Mathematics, The Hebrew University of Jerusalem, Jerusalem, 91904, Israel}
\email{chaim.evenzohar@mail.huji.ac.il}

\author{Sunder Ram Krishnan}
\address{Andrew and Erna Viterbi Faculty of Electrical Engineering, Technion -- Israel Institute of Technology, Haifa 32000, Israel}
\email{eeksunderram@gmail.com}

\thanks{The first and third authors were supported in part by the funding from the European Research Council under Understanding Random Systems via Algebraic Topology (URSAT), ERC Advanced Grant 320422. The second author was supported by BSF 2012188.}

\begin{abstract}
In a previous work, the first and third authors studied a random knot model for all two-bridge knots using billiard table diagrams. Here we present a closed formula for the distribution of the crossing numbers of such random knots. We also show that the probability of any given knot appearing in this model decays to zero at an exponential rate as the length of the billiard table goes to infinity. This confirms a conjecture from the previous work.
\end{abstract}


\maketitle

\section{Introduction}
\label{sec:intro}

There is a recent resurgence of interest in the study of random knots. Several new random models focus on diagrams for which knot invariants can be calculated more easily. The random model in this work is a billiard trajectory with randomly chosen crossings, as introduced in \cite{CoKr} by the first and third authors.

The second author studies other random knot models in \cite{EZ} and in work with Hass, Linial, and Nowik \cite{EZHLN} using petal diagrams of Adams et al. \cite{Adams:petal} and random grid diagrams. Other random knot models can be found in work by Dunfield, Obeidin et al. \cite{Dun:knots}, by Cantarella, Chapman and Mastin~\cite{CanChaMas}, and by Westenberger \cite{West}.  Previous models for random knotting include the closures of braids obtained from random walks in braid groups, and, more prominently, the knotting of random walks in three dimensions.  This last setting was featured in the textbook ``Random knotting and linking'' edited by Millet and Sumners~\cite{Series1994}. The reader can find more specific references in \cite{EZHLN} and~\cite{CoKr}.
	
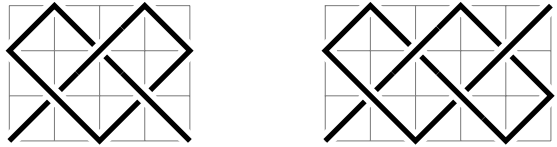
\begin{figure}[h]
\begin{center}
\end{center}
\begin{tikzpicture}[scale=0.6]
\draw[step=1,gray,thin] (0,0) grid (4,3);
\draw[step=1,gray,thin] (6.99,0) grid (12,3);
\draw[white,line width=2.0pt,double=black,double distance=2.0pt] (0,0) -- (1.5,1.5);
\draw[white,line width=2.0pt,double=black,double distance=2.0pt] (2.5,2.5) -- (3,3) -- (4,2) -- (2,0) -- (0,2) -- (1,3) -- (2.5,1.5);
\draw[white,line width=2.0pt,double=black,double distance=2.0pt] (1.5,1.5) -- (2.5,2.5);
\draw[white,line width=2.0pt,double=black,double distance=2.0pt] (2.5,1.5) -- (4,0);
\draw[white,line width=2.0pt,double=black,double distance=2.0pt] (7,0) -- (8.5,1.5);
\draw[white,line width=2.0pt,double=black,double distance=2.0pt] (9.5,2.5) -- (10,3) -- (12,1) -- (11,0) -- (10.5,0.5);
\draw[white,line width=2.0pt,double=black,double distance=2.0pt] (9.5,1.5) -- (8,3) -- (7,2) -- (9,0) -- (12,3);
\draw[white,line width=2.0pt,double=black,double distance=2.0pt] (8.5,1.5) -- (9.5,2.5);
\draw[white,line width=2.0pt,double=black,double distance=2.0pt] (9.5,1.5) -- (10.5,0.5);
\end{tikzpicture}
\caption{Billiard table diagrams for the trefoil and the figure-eight knots.}
\label{trefoil-and-figure-eight}
\end{figure}

Denote by $T(a,b)$ the planar trajectory of a billiard ball in $[0,a]\times[0,b]$ fired at slope one from the lower left corner, where $a$ and $b$ are coprime integers.  \emph{Billiard table diagrams} are obtained from these curves by deciding the crossing information and connecting the two ends outside. See Figure~\ref{trefoil-and-figure-eight} for billiard table diagrams of the trefoil and the figure-eight knots, with $a=3$ for both and $b=4$ and $5$, respectively.

The billiard trajectory $T(a,b)$ is plane isotopic to the curve given as $x(t)=\cos at$ and $y(t)=\cos bt$, where $t \in [0,\pi]$. Such oscillating curves appeared in several natural constructions of knots as parametrized paths in three dimensions. \emph{Harmonic knots} are obtained by closing the open path $(\cos at$, $\cos bt$, $\cos ct)$, where $t \in [0,\pi]$ and $a,b,c$ are pairwise coprime integers~\cite{Comstock}. \emph{Lissajous knots} are parametrized by the already-closed curve $(\cos(at+\phi), \cos(bt+\theta), \cos(ct+\psi))$ where $t \in [0,2 \pi]$, and $\phi,\theta,\psi \in \mathbb{R}$ are fixed phase shifts~\cite{BHJS,JonPrz,HosZir,BDHZ}.

Not all knots are harmonic or Lissajous. However, the diagrams resulting from the projection of Lissajous curves to the $xy$-plane give rise to all knots by suitable choice of the crossing information. This was shown by Lamm~\cite{Lam:dis} in the study of  \emph{Fourier knots}. 

Koseleff and Pecker~\cite{KosPec3} prove a similar statement for the \emph{open} harmonic path $(\cos at,\cos bt)$, in their work on \emph{Chebyshev knots}. They reparametrize it as $(T_a(t),T_b(t))$ using the Chebyshev polynomials $T_n(\cos \theta)=\cos n\theta$, and show that any choice for the crossings can be realized by some $z(t)=T_c(t+\phi)$. Vassiliev represents all knots by polynomials in~\cite{Vas}.

\medskip
The first and third authors of the present paper turn these constructions into a model for random knots \cite{CoKr} by deciding the crossing information $\{\crosszero,\crossone\}^n$ with coin flips independently and uniformly at each crossing of $T(a,b)$. Here we continue with the study of random knots with bridge number at most two, obtained by setting $a=3$ and letting $b=n+1$ where $n=0$ or $1 \bmod 3$. We denote by $D_n$ a random knot diagram in this model and by $K_n$ the resulting random knot.

Much of the literature on random knots focuses on the probability of obtaining a given knot. We progress to study the distribution of the random variable $\text{c}(K_n)$, the crossing number of a random knot. Indeed, it is natural to ask how the number of crossings $n$ in a random billiard table diagram~$D_n$ relates to the minimal number of crossings in any diagram of the resulting knot~$K_n$. In this respect, we prove the following main result.

\medskip
\noindent\textbf{Theorem \ref{cor:main}.}\textit{
The probability that $c \in \{3,\dots,n\}$ is the crossing number of a random two-bridge knot $K_n$ is given by
$$ P\left[\text{\textnormal{c}}\left(K_n \right) = c\right] \;=\;   \sum\limits_{\substack{k \in \{0,\dots,c-2\} \\ c+k \equiv n \text{\textnormal{ (mod 3)}}}} \frac{\binom{c-2}{k}}{2^{n-1}} \cdot F\left(\frac{n-c-k}{3}, \;c+k\right) $$
where 
$$ F(m,\ell) \;:=\; \tfrac{m^2+(\ell+5)m+2}{2} \binom{3m+\ell}{m} - \tfrac{m^2 + (2\ell+9)m + (\ell^2+7\ell+2)}{2} \binom{3m+\ell}{<m} \;.$$
}

\medskip
\noindent\textit{Notation.} 
$\tbinom{n}{<m} := \tbinom{n}{0} + \tbinom{n}{1} + \dots + \tbinom{n}{m-1}$, and $\tbinom{n}{m}:=0$ for $m<0$.
\medskip

This formula implies that the crossing number asymptotically almost surely grows linearly in~$n$. Hence, a random billiard table knot diagram usually cannot be simplified in terms of crossings by more than a multiplicative constant. We explicitly find the asymptotic ratio of this linear growth.

\medskip
\noindent\textbf{Corollary \ref{cor_beta}.}
\textit{The crossing number of a random knot $K_n$ is
$$ \text{\textnormal{c}}\left(K_n\right) \;\;=\;\; (\beta + o(1))n $$
in probability, where $\beta = \frac{\sqrt{5}-1}{4} \approx 0.309 $.
}

\medskip
\noindent\textit{Notation.} 
As usual, $f(n)=o(1)$ means $f(n) \to 0$ as $n \to \infty$. Similarly, a random variable $F_n = o(1)$ \emph{in probability} if~$P[|F_n|>\varepsilon] \to 0$ as $n \to \infty$ for every $\varepsilon>0$.
\medskip

Our analysis of the crossing number distribution relies on Theorem~\ref{thm:main}, where we derive a closed formula for the probability of obtaining any particular two-bridge knot in our random model. A recursive formula for this probability was given in Theorem 4.2 of~\cite{CoKr}. 

As conjectured by the first and third authors in \cite[Conjecture 5.4]{CoKr}, this probability approaches zero as $n$ tends to infinity. Here we answer the subsequent question \cite[Question 5.5]{CoKr} regarding its asymptotic rate. 

\medskip
\vbox{\noindent\textbf{Corollary \ref{prop:main}.}
\textit{The probability of any two-bridge knot $K$ appearing in the random model is
$$ P\left[K_n = K\right] \;\;=\;\; \alpha ^ {\displaystyle(1 + o(1))n} $$
where $\alpha = \sqrt[3]{\frac{27}{32}} \approx 0.945$.
}}

\medskip
\noindent\textbf{Outline. }  In Section \ref{sec:notation}, we provide some background on billiard table diagrams of two-bridge knots and discuss the key ideas in the derivation of our closed formulae. In Section~\ref{sec:insertions} we present the main counting arguments. The results on the distribution of knots
and their crossing numbers are then proved in Sections~\ref{sec:mainthm} and~\ref{sec:crossing}, respectively.

\section{Billiard table diagrams for two-bridge knots}\label{sec:notation}

We first set up some notation and discuss background material on billiard table diagrams for two-bridge knots. We recall a few important definitions from the previous work \cite{CoKr}, give some details on them as required here, and refer the reader to the relevant sections there for further clarifications.

Below we identify a crossing assignment, chosen uniformly from $\{\crosszero,\crossone\}^n$, with a binary \emph{word} in $\{0,1\}^n$, where the crossings are ordered from left to right. For example, in Figure~\ref{trefoil-and-figure-eight}, the diagrams for the trefoil and the figure-eight knots are represented by the words $101$ and $1010$, respectively.  Thus we re-contextualize $P[K_n = K]$ as the probability of certain sequences of ones and zeroes appearing amongst all the $2^n$ possible binary words of length $n$. 

\begin{definition}
\label{def:reductions}
\cite[Definitions 2.6 and 2.8]{CoKr}
(I) An \emph{internal reduction move} is the  deletion of a triple of the type $000$ or $111$ appearing anywhere in the word. (II) An \emph{external reduction move} is the deletion of a triple $001$ or $110$ from the start or a triple $100$ or $011$ from the end of the word.
\end{definition}

See Figure~\ref{internal-and-external-moves} for examples of these moves. Following \cite[Propositions 2.7 and 2.9]{CoKr}, these triples can be deleted without changing the knot type.

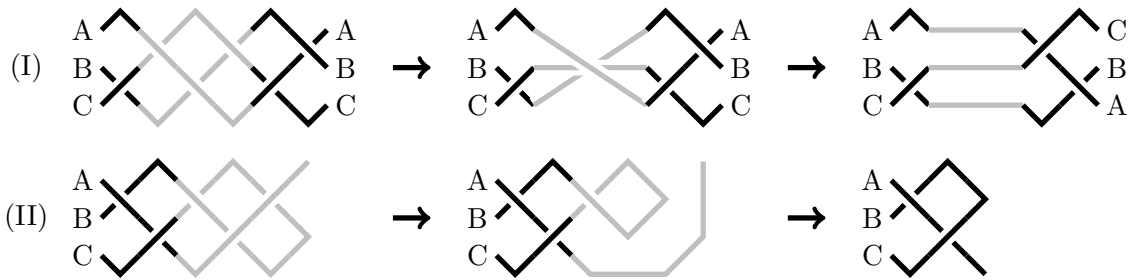
\begin{figure}[h]
\begin{center}
\end{center}
\begin{tikzpicture}[scale=0.5,white,line width=2.0pt,double distance=2.0pt]
\node[black] at (-1,5.5) {(II)};

\begin{scope}[shift={(0,8)}]
\node[black] at (-1,1.5) {(I)};

\node[black] at (0.5,0.5) {C};
\node[black] at (0.5,1.5) {B};
\node[black] at (0.5,2.5) {A};
\draw[double=black] (1,2.5) -- (1.5,3) -- (2,2.5);
\draw[double=black] (1,1.5) -- (2,0.5);
\draw[double=black] (1,0.5) -- (2,1.5);
\draw[double=lightgray] (2,0.5) -- (2.5,0) -- (3,0.5);
\draw[double=lightgray] (2,1.5) -- (3,2.5);
\draw[double=lightgray] (2,2.5) -- (3,1.5);
\draw[double=lightgray] (3,2.5) -- (3.5,3) -- (4,2.5);
\draw[double=lightgray] (3,0.5) -- (4,1.5);
\draw[double=lightgray] (3,1.5) -- (4,0.5);
\draw[double=lightgray] (4,0.5) -- (4.5,0) -- (5,0.5);
\draw[double=lightgray] (4,1.5) -- (5,2.5);
\draw[double=lightgray] (4,2.5) -- (5,1.5);
\draw[double=black] (5,2.5) -- (5.5,3) -- (6,2.5);
\draw[double=black] (5,1.5) -- (6,0.5);
\draw[double=black] (5,0.5) -- (6,1.5);
\draw[double=black] (6,0.5) -- (6.5,0) -- (7,0.5);
\draw[double=black] (6,1.5) -- (7,2.5);
\draw[double=black] (6,2.5) -- (7,1.5);
\node[black] at (7.5,0.5) {C};
\node[black] at (7.5,1.5) {B};
\node[black] at (7.5,2.5) {A};
\draw[->,black] (8.75,1.5) -- (9.75,1.5);
\node[black] at (11,0.5) {C};
\node[black] at (11,1.5) {B};
\node[black] at (11,2.5) {A};
\draw[double=black] (11.5,2.5) -- (12,3) -- (12.5,2.5);
\draw[double=black] (11.5,1.5) -- (12.5,0.5);
\draw[double=black] (11.5,0.5) -- (12.5,1.5);
\draw[double=lightgray] (12.5,0.5) -- (15.5,2.5);
\draw[double=lightgray] (12.5,1.5) -- (15.5,1.5);
\draw[double=lightgray] (12.5,2.5) -- (15.5,0.5);
\draw[double=black] (15.5,2.5) -- (16,3) -- (16.5,2.5);
\draw[double=black] (15.5,1.5) -- (16.5,0.5);
\draw[double=black] (15.5,0.5) -- (16.5,1.5);
\draw[double=black] (16.5,0.5) -- (17,0) -- (17.5,0.5);
\draw[double=black] (16.5,1.5) -- (17.5,2.5);
\draw[double=black] (16.5,2.5) -- (17.5,1.5);
\node[black] at (18,0.5) {C};
\node[black] at (18,1.5) {B};
\node[black] at (18,2.5) {A};
\draw[->,black] (19.25,1.5) -- (20.25,1.5);
\node[black] at (21.5,0.5) {C};
\node[black] at (21.5,1.5) {B};
\node[black] at (21.5,2.5) {A};
\draw[double=black] (22,2.5) -- (22.5,3) -- (23,2.5);
\draw[double=black] (22,1.5) -- (23,0.5);
\draw[double=black] (22,0.5) -- (23,1.5);
\draw[double=lightgray] (23,0.5) -- (25.5,0.5);
\draw[double=lightgray] (23,1.5) -- (25.5,1.5);
\draw[double=lightgray] (23,2.5) -- (25.5,2.5);
\draw[double=black] (25.5,0.5) -- (26,0) -- (26.5,0.5);
\draw[double=black] (25.5,2.5) -- (26.5,1.5);
\draw[double=black] (25.5,1.5) -- (26.5,2.5);
\draw[double=black] (26.5,2.5) -- (27,3) -- (27.5,2.5);
\draw[double=black] (26.5,0.5) -- (27.5,1.5);
\draw[double=black] (26.5,1.5) -- (27.5,0.5);
\node[black] at (28,0.5) {A};
\node[black] at (28,1.5) {B};
\node[black] at (28,2.5) {C};

\end{scope}

\node[black] at (0.5,4.5) {C};
\node[black] at (0.5,5.5) {B};
\node[black] at (0.5,6.5) {A};

\begin{scope}[shift={(.5,0)}]

\draw[double=black] (0.5,4.5) -- (1,4) -- (1.5,4.5);
\draw[double=black] (0.5,5.5) -- (1.5,6.5);
\draw[double=black] (0.5,6.5) -- (1.5,5.5);
\draw[double=black] (1.5,6.5) -- (2,7) -- (2.5,6.5);
\draw[double=black] (1.5,5.5) -- (2.5,4.5);
\draw[double=black] (1.5,4.5) -- (2.5,5.5);
\draw[double=lightgray] (2.5,4.5) -- (3,4) -- (3.5,4.5);
\draw[double=lightgray] (2.5,5.5) -- (3.5,6.5);
\draw[double=lightgray] (2.5,6.5) -- (3.5,5.5);
\draw[double=lightgray] (3.5,6.5) -- (4,7) -- (4.5,6.5);
\draw[double=lightgray] (3.5,5.5) -- (4.5,4.5);
\draw[double=lightgray] (3.5,4.5) -- (4.5,5.5);
\draw[double=lightgray] (4.5,4.5) -- (5,4) -- (5.5,4.5);
\draw[double=lightgray] (4.5,6.5) -- (5.5,5.5);
\draw[double=lightgray] (4.5,5.5) -- (5.5,6.5);
\draw[double=lightgray] (5.5,6.5) -- (6,7);
\draw[double=lightgray] (5.5,5.5) -- (6,5) -- (5.5,4.5);

\end{scope}

\node[black] at (11,4.5) {C};
\node[black] at (11,5.5) {B};
\node[black] at (11,6.5) {A};

\begin{scope}[shift={(1.5,0)}]

\draw[->,black] (7.25,5.5) -- (8.25,5.5);

\end{scope}

\node[black] at (21.5,4.5) {C};
\node[black] at (21.5,5.5) {B};
\node[black] at (21.5,6.5) {A};

\begin{scope}[shift={(2,0)}]

\draw[double=black] (9.5,4.5) -- (10,4) -- (10.5,4.5);
\draw[double=black] (9.5,5.5) -- (10.5,6.5);
\draw[double=black] (9.5,6.5) -- (10.5,5.5);
\draw[double=black] (10.5,6.5) -- (11,7) -- (11.5,6.5);
\draw[double=black] (10.5,5.5) -- (11.5,4.5);
\draw[double=black] (10.5,4.5) -- (11.5,5.5);
\draw[double=lightgray] (11.5,4.5) -- (12,4) -- (14,4) -- (15,5) -- (15,7);
\draw[double=lightgray] (11.5,5.5) -- (12.5,6.5);
\draw[double=lightgray] (11.5,6.5) -- (12.5,5.5);
\draw[double=lightgray] (12.5,6.5) -- (13,7) -- (14,6) -- (13,5) -- (12.5,5.5);

\end{scope}

\begin{scope}[shift={(3,0)}]

\draw[->,black] (16.25,5.5) -- (17.25,5.5);

\end{scope}

\begin{scope}[shift={(3.5,0)}]

\draw[double=black] (18.5,4.5) -- (19,4) -- (19.5,4.5);
\draw[double=black] (18.5,5.5) -- (19.5,6.5);
\draw[double=black] (18.5,6.5) -- (19.5,5.5);
\draw[double=black] (19.5,6.5) -- (20,7) -- (20.5,6.5);
\draw[double=black] (19.5,5.5) -- (20.5,4.5);
\draw[double=black] (19.5,4.5) -- (20.5,5.5);
\draw[double=black] (20.5,4.5) -- (21,4);
\draw[double=black] (20.5,5.5) -- (21,6) -- (20.5,6.5);

\end{scope}

\end{tikzpicture}
\caption{Examples for the two reduction moves in billiard table knot diagrams: \\ (I)~internal, deleting the triple $111$ from ``...011101...''; \\ (II)~external, deleting the suffix $100$ from ``...10100''. \\ Note that in (I) we half-twist the whole right hand side of the knot.}
\label{internal-and-external-moves}
\end{figure}

After performing all possible reduction moves, we end up with a unique reduced word whose form is defined below.

\begin{definition}
\label{def:reduced}
We say that a \emph{run} is a maximal sequence of one or more consecutive identical letters. A word that is composed of runs only of the form $0$, $1$, $00$, and $11$ is called \emph{reduced with respect to internal moves}. Such a word that also starts with $01$ or $10$ and ends with $01$ or $10$ is simply called \emph{reduced}. For example, the word $1010010$ is reduced with six runs.
\end{definition}

These notions assist in a simple description of the crossing numbers of two-bridge knots. 

\begin{proposition}[immediate from \cite{KosPec4}, Proposition 2.5]\label{cross}
The crossing number of the knot corresponding to a reduced word is equal to the number of runs in that word.
\end{proposition}

\begin{proof} 
Consider a billiard table knot diagram that corresponds to a reduced word. A pair of consecutive identical crossings can be replaced with a single one, as demonstrated for the word $1010010$ in Figure~\ref{two-to-one}. Note that the resulting diagram is no longer a billiard table diagram. 

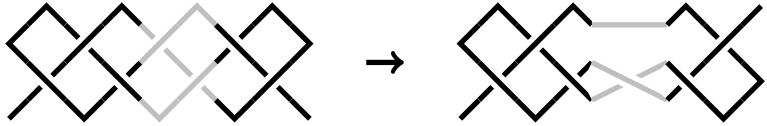
\begin{figure}[h]
\begin{center}
\end{center}
\begin{tikzpicture}[scale=0.5,white,line width=2.0pt,double distance=2.0pt]
\draw[double=black] (1,1.5) -- (0.5,2) -- (1,2.5);
\draw[double=black] (0.5,0) -- (1,0.5);
\draw[double=black] (1,2.5) -- (1.5,3) -- (2,2.5);
\draw[double=black] (1,0.5) -- (2,1.5);
\draw[double=black] (1,1.5) -- (2,0.5);
\draw[double=black] (2,0.5) -- (2.5,0) -- (3,0.5);
\draw[double=black] (2,2.5) -- (3,1.5);
\draw[double=black] (2,1.5) -- (3,2.5);
\draw[double=black] (3,2.5) -- (3.5,3) -- (4,2.5);
\draw[double=black] (3,0.5) -- (4,1.5);
\draw[double=black] (3,1.5) -- (4,0.5);
\draw[double=lightgray] (4,0.5) -- (4.5,0) -- (5,0.5);
\draw[double=lightgray] (4,2.5) -- (5,1.5);
\draw[double=lightgray] (4,1.5) -- (5,2.5);
\draw[double=lightgray] (5,2.5) -- (5.5,3) -- (6,2.5);
\draw[double=lightgray] (5,1.5) -- (6,0.5);
\draw[double=lightgray] (5,0.5) -- (6,1.5);
\draw[double=black] (6,0.5) -- (6.5,0) -- (7,0.5);
\draw[double=black] (6,1.5) -- (7,2.5);
\draw[double=black] (6,2.5) -- (7,1.5);
\draw[double=black] (7,2.5) -- (7.5,3) -- (8,2.5);
\draw[double=black] (7,1.5) -- (8,0.5);
\draw[double=black] (7,0.5) -- (8,1.5);
\draw[double=black] (8,1.5) -- (8.5,2) -- (8,2.5);
\draw[double=black] (8,0.5) -- (8.5,0);
\draw[->,black] (10,1.5) -- (11,1.5);
\draw[double=black] (13,1.5) -- (12.5,2) -- (13,2.5);
\draw[double=black] (12.5,0) -- (13,0.5);
\draw[double=black] (13,2.5) -- (13.5,3) -- (14,2.5);
\draw[double=black] (13,0.5) -- (14,1.5);
\draw[double=black] (13,1.5) -- (14,0.5);
\draw[double=black] (14,0.5) -- (14.5,0) -- (15,0.5);
\draw[double=black] (14,2.5) -- (15,1.5);
\draw[double=black] (14,1.5) -- (15,2.5);
\draw[double=black] (15,2.5) -- (15.5,3) -- (16,2.5);
\draw[double=black] (15,0.5) -- (16,1.5);
\draw[double=black] (15,1.5) -- (16,0.5);
\draw[double=lightgray] (16,0.5) -- (18,1.5);
\draw[double=lightgray] (16,1.5) -- (18,0.5);
\draw[double=lightgray] (16,2.5) -- (18,2.5);
\draw[double=black] (18,2.5) -- (18.5,3) -- (19,2.5);
\draw[double=black] (18,0.5) -- (19,1.5);
\draw[double=black] (18,1.5) -- (19,0.5);
\draw[double=black] (19,0.5) -- (19.5,0) -- (20,0.5);
\draw[double=black] (19,2.5) -- (20,1.5);
\draw[double=black] (19,1.5) -- (20,2.5);
\draw[double=black] (20,1.5) -- (20.5,1) -- (20,0.5);
\draw[double=black] (20,2.5) -- (20.5,3);
\end{tikzpicture}
\caption{Replacing a run of two crossings by one, with a half-twist of the right hand side of the knot.}
\label{two-to-one}
\end{figure}

By repeating this move for all runs of length two, we reduce the number of crossings in the knot diagram down to the number of runs in the word. One can easily verify that the resulting diagram is \emph{alternating}, which means that over-crossings and under-crossings occur alternately as one travels along it. It is also \emph{reduced} in the sense that no crossing is an \emph{isthmus} between two parts of the diagram. By the classical Tait conjectures, such a diagram realizes the crossing number~\cite{KauffmanBracket,Murasugi,Th}.
\end{proof}

One might wonder about the number of ways in which a given knot can be written in reduced form. Before turning to this issue, we make the following remark.
\begin{remark}
\label{property:symmetry}

There are three symmetries of the reduced word representation of two-bridge knots implicit in~\cite{KauffLamb:class}, \cite{KosPec4} and \cite{CoKr}.

\begin{enumerate}
\item The \emph{complement} word, obtained by flipping every digit, represents a mirror image of the knot. For example, $101$ is the left-handed trefoil and $010$ is the right-handed trefoil. If the knot is achiral, then we obtain two different representations of the same knot.
\item The \emph{reverse} word, e.g. $01001 \leftrightarrow 10010$, yields the same knot with its orientation reversed. Since all two-bridge knots are \emph{invertible}, this operation may further double the number of representations, as long as the reverse is different from the original word and its complement.
\item We resize every internal run in $w$ from length two to one and vice versa. For example, $010110 \leftrightarrow 0110010$. These changes correspond to the operation from Figure~\ref{two-to-one} and its inverse and indeed yield a billiard table diagram of the mirror image of the original knot.
\end{enumerate}
\end{remark}

It can be shown by the theory of terminated $\pm 1$ continued fraction expansions for rational knots that every two-bridge knot $K$ has a reduced word representation, unique except for these three simple operations~\cite[Theorem 2.19 and Lemma 2.20]{CoKr}. 

The third operation lets us switch between two different \emph{reduced lengths}, $\ell_0(K)$ and $\ell_1(K)$, where $\ell_i(K) \equiv i$ mod~$3$. The first two operations determine a \emph{multiplicity} $r(K)$, which is the number of ways to write a two-bridge knot $K$ in each reduced length. Taking symmetries into account, $r(K) \in \{1,2\}$ or $r(K) \in \{2,4\}$, depending up on whether one decides to distinguish between the two mirror images of chiral knots. Ernst and Sumners~\cite{ErnSum} include a discussion on chirality in their results on the number of two-bridge knots of a given crossing number.

Since the first and last runs in a reduced word consist of one letter each, and all the others of one or two letters, by Proposition~\ref{cross} we see that the following holds true.

\begin{corollary}\label{moshe}
A two-bridge knots $K$ with crossing number $c$ have reduced lengths $\ell_0(K)$ and $\ell_1(K)$ in the range $\{c,c+1,\dots,2c-2\}$ such that $\ell_i(K) \equiv i$ mod $3$.
\end{corollary}

\subsection*{Main proof idea}

With the background material behind us, we turn now to the presentation of the key ideas used in proving our main results.

The proof of the new closed formula for the knot distribution in Theorem~\ref{thm:main} proceeds via counting words in $\{0,1\}^n$ that represent a given knot $K$. In this work, rather than performing reduction moves as in \cite{CoKr}, we start from a reduced word and count the number of $n$-letter words obtainable by inserting triples back into it.

As in~\cite{CoKr}, we count external and internal insertions in separate stages considered in Subsections~\ref{sec:internal} and~\ref{sec:external}, respectively. The main contribution to the count comes from internal insertions. These triples can be inserted in an exponential number of ways, while the external ones introduce at most a quadratic factor as demonstrated in the proof of Corollary~\ref{prop:main}. In particular, we show in Proposition~\ref{count-I} that sequences of $m$ internal insertions into an $\ell$-letter word result in exactly $\tbinom{3m+\ell}{m} - \tbinom{3m+\ell}{<m}$ different words.

The main tool in the proof of this central proposition is the \emph{location map} $\Lambda$, defined on every word~$w'$ obtained from a  word~$w$ of length $\ell$ by  $m$ internal insertions. This explicit definition requires the choice of a canonical sequence of insertions that takes~$w$ to~$w'$. The image $\Lambda(w')$ is then a subset of $\{1,\dots,3m+\ell\}$ of size at most~$m$ that essentially records the locations at which these internal insertions take place.

Since the location map $\Lambda$ is shown to be injective, the main challenge remains to characterize the subsets $L \subseteq \{1,\dots,3m+\ell\}$ that come from insertion locations. To that end, we introduce an algorithm that tries to reconstruct a word $w'$ starting from a word $w$ and a set of locations~$L$, such that $\Lambda(w')=L$. Thus the counting of obtainable words $w'$ boils down to the analysis of the reconstruction algorithm and the number of possible inputs $L$ on which it succeeds. 

It turns out that those successful location sets $L$ of a given size can be characterized by a classical combinatorial result known as \emph{Bertrand's Ballot Theorem}. The precise number of words of length $n$ representing $K$ in Theorem~\ref{thm:main} is then obtained by summation over the size of~$L$, over external insertions, and over the reduced words that represent~$K$. The distribution of the crossing number in Theorem~\ref{cor:main} is deduced from it using Proposition~\ref{cross}.

\section{Counting insertions}
\label{sec:insertions}

\subsection{Counting internal insertions}
\label{sec:internal}

As described above, the closed formula regarding the distribution of knots in Theorem~\ref{thm:main} is derived by counting different \mbox{$n$-letter} words obtainable from a reduced $\ell$-letter word by inserting triples back. The main contribution comes from internal insertions, whose counting is the topic of this subsection.

\begin{definition}
\label{def:insertions}
Given a word $w \in \{0,1\}^{\ell}$ and an integer $m \geq 0$, the \emph{insertion set} $I(w,m)$ is the set of all words $w'$ obtained from $w$ by $m$ insertions of triples. 
\end{definition}

\begin{definition}
\label{def:internal}
An \emph{internal insertion} takes a word in $\{0,1\}^k$ to a word in $\{0,1\}^{k+3}$ by inserting a $000$ or a $111$ at any of the $k+1$ possible locations, including both ends. Given $w \in \{0,1\}^{\ell}$ and $m \geq 0$, the \emph{internal insertion set} $I'(w,m)$ is the set of all words in $\{0,1\}^{3m+\ell}$ obtained from $w$ by $m$ internal insertions.
\end{definition}

\begin{example} We see that
$100001001110 \in I'(101,3)$ by three insertions: $1(000)01(00(111)0)$.
\end{example}

It is immediate from the definitions that every $w' \in I'(w,m)$ can be reduced to $w$ by $m$ reduction moves and hence represents the same knot as $w$. Recalling the notation $\tbinom{n}{<m}$ introduced in Section~\ref{sec:intro}, we can now state the main result of this section.

\begin{proposition}
\label{count-I}
The number of words obtained from $w \in \{0,1\}^{\ell}$ by $m \geq 0$ internal insertions is
$$ \left|I'(w,m)\right| \;=\; \binom{3m+\ell}{m} - \binom{3m+\ell}{<m}\;.$$ 
\end{proposition}

\begin{remark}
In the forthcoming analysis of $I'(w,m)$ we do not require that $w$ is reduced with respect to internal moves, although this will be the setting in the application.
\end{remark}

Of course, there might be several different sequences of internal insertions that lead from one word to another. To avoid double counting, we choose a canonical way to perform these insertions. This necessitates the introduction of some terminology.

\begin{definition}
\label{def:moshe}
An insertion of a triple into $w \in \{0,1\}^k$ takes place \emph{at location} $j \in \{1,\dots,k+1\}$ if the new triple appears between the first $j-1$ letters of $w$ and the other $k-j+1$ letters. We say that the insertion takes place either \emph{before} $0$ or \emph{before} $1$, depending up on the $j$th letter of $w$, which then appears right after the inserted triple. The only exception is $j=k+1$, where there is no such letter and the insertion takes place \emph{at the end} of $w$. 
\end{definition}

The canonical way to insert internal moves promised above is then given by the following lemma.

\begin{lemma}
\label{wloglemma}
Let $w \in \{0,1\}^{\ell}$ and $w' \in I'(w,m)$. Then $w'$ is obtainable from $w$ by a unique sequence of insertions at increasing locations such that $000$ is inserted either before $1$ or at the end, and $111$ is inserted either before $0$ or at the end.
\end{lemma}

\begin{proof}
Inserting a triple at location $j$ followed by another in location $i$ where $i \leq j$, is equivalent to inserting the latter triple first at location $i$ and the former triple thereafter at location $j+3$. We may perform these swaps as long as there exist such unsorted pairs of adjacent insertions.

Secondly, if a triple $000$ is inserted at location $i$ before $0$, then inserting $000$ at location $i+1$ results in the same word. We keep incrementing locations in the sequence as long as the property stated in the lemma is violated.

Note that both types of corrections strictly increase the sum of locations of insertions in the sequence. Since this sum is bounded, say by $m(3m+\ell)$, the process must terminate with a sequence that satisfies both conditions.

To show uniqueness of the sequence of insertions, consider the first insertion in which two such sequences differ. If the two insertions take place at different locations $i<j$, then the resulting words differ in the $i$th letter, by the condition that $000$ is inserted before $1$ and vice versa. If at the end of a $k$-letter intermediate word a $000$ was inserted in one case and a $111$ in the other, then the two resulting words differ in the $(k+1)$th letter, as further insertions only take place at locations $k+2$ and greater. 
\end{proof}

We continue as outlined at the end of Section~\ref{sec:notation}. Fix $w \in \{0,1\}^{\ell}$ and $m \geq 0$. In order to count words in $I'(w,m)$, we map them to subsets of $[3m+\ell] := \{1,2,\dots,3m+\ell\}$. Denote by $\tbinom{[3m+\ell]}{\leq m}$ the collection of all subsets of $[3m+\ell]$ of size at most $m$. We define the \emph{location map} to be
$$ \Lambda : I'(w,m) \;\to\; \binom{[3m+\ell]}{\leq m}\;.$$
The image $\Lambda(w')$ for any given $w' \in I'(w,m)$ is defined as follows. Consider the canonical sequence of insertions given by  Lemma~\ref{wloglemma}, reconstructing $w'$ from $w$. We record in $\Lambda(w')$ the locations of all insertions in that sequence, except in the case where a $000$ is inserted at the end of the word. Clearly $\Lambda(w') \subseteq [3m+\ell]$ of size at most $m$. 

\begin{lemma}
\label{injective}
For every word $w$ and $m \geq 0$, the map $\Lambda$ is injective.
\end{lemma}

\begin{proof}
Let $w'$ and $w''$ be two words such that $\Lambda(w') \cap [i] = \Lambda(w'') \cap [i]$, where as usual $[i] = \{1,\dots,i\}$. We show by induction on $i$ that up to location $i$ the two words have the same canonical sequence of insertions. By the same argument as in the previous lemma, it would then follow that they agree on the first $i$ letters.

By the induction hypothesis, we assume that the two canonical sequences of insertions agree up to location $i-1$. We perform these insertions to $w$ and obtain some intermediate word. Now there are two cases.

If the intermediate word has length $i-1$, then we can either insert a $000$ or a $111$ at the end of it. Since $i \in \Lambda(w')$ if and only if $i \in \Lambda(w'')$, it is the same insertion in both sequences, by the definition of $\Lambda$ for insertions at the end.

Otherwise, there are at least $i$ letters in the intermediate word. Suppose that the $i$th letter is, say, $0$. Our options at location $i$ are either to insert $111$ or to have no insertion. Again, since $i \in \Lambda(w')$ if and only if  $i \in \Lambda(w'')$, it is the same choice in both words.

Finally, letting $i=3m+\ell$ we obtain the result that if $\Lambda(w') = \Lambda(w'')$, then $w'=w''$.
\end{proof}

Now there is a bijection between $I'(w,m)$ and $\text{Image}(\Lambda) \subseteq \tbinom{[3m+\ell]}{\leq m}$. 

To arrive at the number of subsets $L \subseteq$ $[3m+\ell]$ that have a preimage, we study the inverse function. We describe $\Lambda^{-1}$ as an algorithm that reconstructs a word $w' \in I'(w,m)$ such that $\Lambda(w')=L$. For the purpose of storing letters that are expected to be written, the algorithm uses a \emph{stack} $S$, which is a data structure in which the last element that was pushed in is the first to pop out. It will be convenient to assume that $S$ yields $0$ if one tries to peek at its top or to pop out an element while it is empty.

\begin{samepage}
\begin{algorithm}[Reconstruction of $w'$]
\label{alg}
Input: $w \in \{0,1\}^{\ell}$, $m \in \mathbb{N}$, $L \in \tbinom{[3m+\ell]}{\leq m}$ 
\begin{itemize}
	\item[$\blacktriangleright$]
	for each location $j \in \{\ell,\dots,2,1\}$
	\begin{itemize}
		\item[$\blacktriangleright$]
		push the $j$th letter of $w$ into $S$
	\end{itemize}
	\item[$\blacktriangleright$] 
	for each location $i \in \{1,2,\dots,3m+\ell\}$
	\begin{itemize}
		\item[$\blacktriangleright$] 
		if $i \in L$
		\begin{itemize}
			\item[$\blacktriangleright$]
			peek at the letter at the top $S$
			\item[$\blacktriangleright$]
			push into $S$ three copies of the other letter
		\end{itemize}
		\item[$\blacktriangleright$] 
		pop one letter out of $S$
		\item[$\blacktriangleright$]
		write it down as the $i$th letter of $w'$
	\end{itemize}
	\item[$\blacktriangleright$]
	if $S$ is empty
	\begin{itemize}
		\item[$\blacktriangleright$]
		output $w'$
	\end{itemize}
	\item[$\blacktriangleright$]
	otherwise, report failure
\end{itemize}
\end{algorithm}
\end{samepage}

\begin{table}[b]
\centering
\begin{tabular}{ccc}
\begin{tabular}{|c|c|l|r|}
\hline
$i$ & $\in L$ & word $w'$ & stack $S$ \\ \hline \hline
- & - & & 101 \\ \hline
1 & \checkmark & 0 & 00101 \\ \hline
2 & & 00 & 0101 \\ \hline
3 & & 000 & 101 \\ \hline
4 & & 0001 & 01 \\ \hline
5 & \checkmark & 00011 & 1101 \\ \hline
6 & & 000111 & 101 \\ \hline
7 & & 0001111 & 01 \\ \hline
8 & & 00011110 & 1 \\ \hline
9 & & 000111101 & \\ \hline
\end{tabular}
& &
\begin{tabular}{|c|c|l|r|}
\hline
$i$ & $\in L$ & word $w'$ & stack $S$ \\ \hline \hline
- & - & & 101 \\ \hline
1 & \checkmark & 0 & 00101 \\ \hline
2 & & 00 & 0101 \\ \hline
3 & & 000 & 101 \\ \hline
4 & & 0001 & 01 \\ \hline
5 & & 00010 & 1 \\ \hline
6 & & 000101 & \\ \hline
7 & & 0001010 & \\ \hline
8 & \checkmark & 00010101 & 11 \\ \hline
9 & & 000101011 & 1 \\ \hline
\end{tabular} \\ 
& & \\
$\{1,5\} \in \Lambda[I'(101,2)]$ & & $\{1,8\} \not\in \Lambda[I'(101,2)]$
\\ & &
\end{tabular}
\caption{Algorithm~\ref{alg} illustrated for Examples \ref{ex:recalgo1} and \ref{ex:recalgo2}.}
\label{tab:ex}
\end{table}

\begin{lemma}
\label{correct}
Algorithm~\ref{alg} is correct, that is, if $L \in \text{Image}(\Lambda)$, then $w' \in I'(w,m)$ and $\Lambda(w')=L$.
\end{lemma}

\begin{proof}
One can observe that the algorithm imitates the process of inserting triples into $w$ at locations $L$. As long as there are no insertions, we write down $w$. Also, at each insertion, we give priority to the inserted triple and defer letters of $w$ and other unfinished triples for later. At each step, the written part of $w'$, together with the part that waits in the stack, form an intermediate word that includes the insertions performed so far.

We decide whether to insert a triple $000$ or $111$ by the condition that it is different from the next letter. If the stack empties before the location $i$ reaches the end of $w'$, then we have written all we could and some triple must be inserted. By default we assume it to be $000$ unless $i \in L$, in which case we insert $111$. This arbitrary choice suits the definition of $\Lambda$ that discards insertions of $000$ at the end. It also complies with the convention that an empty stack pops $0$. Note that the number of such steps is clearly divisible by three. In other words, an equivalent algorithm could push $000$ whenever one tries to pop out of an empty stack.

We check that the simulated sequence of insertions is canonical as in Lemma~\ref{wloglemma}. Since these insertions happen at location $i$ whenever $i \in L$, and since $i$ is increasing, so is the order of the insertions. Except possibly for insertions at the end, each inserted triple is different from the letter at the top of the stack, which is indeed the next letter to be written at that point. 

In conclusion, the word $w'$ returned by the algorithm is obtained from $w$ by a canonical sequence of insertions at the locations in $L$ together with some insertions of $000$ at the end, exactly as in the definition of $\Lambda$.
\end{proof}

It is clear from the proof of Lemma~\ref{correct} that every $L \in \tbinom{[3m+\ell]}{\leq m}$ for which the algorithm does not fail comes from some legitimate word $w' \in I'(w,m)$. Towards analyzing the image of the location map~$\Lambda$, we thus see that it is enough to characterise those subsets for which the algorithm fails. To illustrate, we provide examples for both cases.

\begin{example}
\label{ex:recalgo1}
Let $w=101$, $m=2$, and $L=\{1,5\}$. Then the reconstruction algorithm produces a preimage $w' \in I'(101,2)$ as shown in Table \ref{tab:ex} on the left, and indeed $L = \Lambda(000111101)$. 
\end{example}
\begin{example}
\label{ex:recalgo2}
However, $L=\{1,8\} \not \in \text{Image}(\Lambda)$ for $I'(101,2)$ as the stack is not empty at the end of the algorithm, as shown in Table \ref{tab:ex} on the right.
\end{example}

From Examples \ref{ex:recalgo1} and \ref{ex:recalgo2}, we see what might prevent a subset $L$ from having a preimage under $\Lambda$. Loosely speaking, the stack might not be empty after step $3m+\ell$ if too many elements of $L$ are too close to the end of the interval $[3m+\ell]$. Using the following classical result from enumerative combinatorics, we make this condition precise in the subsequent proof.

\smallskip \noindent {\bf Bertrand's Ballot Theorem.} (e.g.~\cite{ballot}, pp. 350-351) {\it 
Suppose that in an election, candidate $A$ receives $a$ votes and candidate $B$ receives $b$ votes, where $b \geq ra$ for some positive integer $r$. \nopagebreak \\ Then the number of ways the ballots can be ordered so that $B$ maintains at least $r$ times as many votes as $A$ throughout the counting of the ballots is given by $\frac{b+1-ra}{b+1}\binom{a+b}{a}$.
} \smallskip

\begin{proof}[Proof of Proposition~\ref{count-I}]
By the injectivity in Lemma~\ref{injective}, it is enough to count $\text{Image}(\Lambda)$ to arrive at $|I'(w,m)|$. For this, we presently characterise the subsets $L \in \tbinom{[3m+\ell]}{\leq m}$ on which Algorithm~\ref{alg} fails and those on which it does not.

At each step $i \in L$, the stack size increases by $2$, while $i \not\in L$ causes the size to decrease by $1$. Therefore, if for some suffix of $[3m+\ell]$, the ratio between elements in $L$ and elements not in $L$ is more than $1$:$2$, then the stack becomes larger over that interval and cannot end up empty when Algorithm~\ref{alg} terminates. 

Conversely, suppose that the algorithm failed as the stack did not end up empty. The total number of letters pushed into the stack equals $\ell+3|L| \leq \ell+3m$, and at each of the $\ell+3m$ steps we tried to pop a letter out. Therefore, there must have been some intermediate step at which a letter could not pop out as the stack was empty. This means that the stack became larger over some suffix of $[3m+\ell]$ in which the ratio of elements in $L$ to those not in $L$ was more than $1$:$2$.

In conclusion, $|\text{Image}(\Lambda)|$ is the number of subsets $L$ of $[3m+\ell]$ of size at most $m$ such that in every suffix of $[3m+\ell]$, the number of elements not in $L$ is at least twice as much as those in $L$. 

We allude to the Bertrand's Ballot Theorem to count these subsets, finding analogous quantities in the problem at hand. Here the numbers $(3m+\ell,\dots,2,1)$ are ``voting'' for (A) being in $L$, or (B) not being in $L$. Thus the number of votes are $a=|L|$ and $b=3m+\ell-|L|$, and the factor by which candidate $B$ should lead throughout the counting is $r=2$. For each $|L| \in \{0,1,\dots,m\}$, the number of feasible location sets $L$ of size $|L|$ is therefore equal to
\begin{equation}
 \frac{3m+\ell-3|L|+1}{3m+\ell-|L|+1}\binom{3m+\ell}{|L|}\;=\;\binom{3m+\ell}{|L|}-2\binom{3m+\ell}{|L|-1}\;. 
 \label{eq:needlater}
 \end{equation}
Summing the above over $|L| \in \{0,1,\dots,m\}$, we obtain 
$$ |\text{Image}(\Lambda)| \;=\; \binom{3m+\ell}{\leq m}-2\binom{3m+\ell}{\leq m-1}\;=\;\binom{3m+\ell}{m}-\binom{3m+\ell}{< m} $$
as required. 
\end{proof}

\subsection{Counting external insertions}
\label{sec:external}

We next extend the counting to the insertion set $I(w,m)$, which comprises of all words that are obtained from a reduced word $w$ by $m$ insertions, both internal and external.  Before embarking on this venture, we make a definition analogous to Definition \ref{def:internal}.

\begin{definition}
\label{def:external}
An \emph{external insertion} takes a $k$-letter word to a $(k+3)$-letter word, either by adding a prefix $001$ or $110$, or by adding a suffix $011$ or $100$.

Given $w \in \{0,1\}^{\ell}$ and $m \geq 0$, the \emph{external insertion set} $I''(w,m)$ is the set of all words in $\{0,1\}^{3m+\ell}$ obtained from $w$ by $m$ external insertions.
\end{definition}

As explained in Section \ref{sec:notation}, any word $w'' \in I''(w,m)$ represents the same two-bridge knot as~$w$. The following lemma enables us to insert triples in two separate stages: first all of the external insertions followed by all of the internal insertions.

\begin{lemma}
\label{wloglemma2}[in the spirit of~\cite{CoKr}, p. 16]
Let $w' \in I(w,m)$ where $w \in \{0,1\}^{\ell}$. Then there exists some $e \in \{0,\dots,m\}$ and intermediate word $w'' \in \{0,1\}^{\ell+3e}$ such that $w'' \in I''(w,e)$ and \mbox{$w' \in I'(w'',m-e)$}. 

Moreover, there exists such $w''$ of the form $w''=p^iws^j$, with a prefix \mbox{$p \in \{001,110\}$} and a suffix $s \in \{011,100\}$, where $i+j=e$. 

Furthermore, if $w$ is reduced, then this representation is unique.
\end{lemma}

\begin{proof}
Observe that whenever an external insertion follows an internal one, the order of the insertions can be easily swapped. Hence we can rearrange the sequence such that all external insertions happen first.  

For the second property, note that two consecutive different prefixes are equivalent to two internal insertions, as in $(001)(110) = (00(111)0)$. Thus as long as there are adjacent unequal external insertions, we replace them by internal ones. Eventually, all prefixes would be the same and similarly for suffixes. This yields $w''=p^iws^j$ with $i+j=e$ external insertions, as required.

As for uniqueness, note that $w''$ is reduced with respect to internal moves, as it is only composed of one- and two-letter runs from $\{0,00\}$ and $\{1,11\}$ alternately. Indeed, this follows from $p^i$, $w$, and $s^j$ having this property, while only one-letter runs occur at the ends where they are concatenated. It follows that $w''$ is uniquely determined by~$w'$. Now, since $w$ starts and ends with $01$ and $10$, it is easy to read off all prefixes and suffixes from~$w''$.
\end{proof}

\begin{remark}
\label{smalll}
The four short words $\{0, 1, 00, 11\}$ are not reduced by our definition, since they do not start and end with $01$ or $10$. Indeed, some of the uniqueness arguments do not hold for them. For example, $0011$ reduces both to $0$ and $1$, and $00100$ reduces to $00$ in two different ways. One may verify though that Lemma~\ref{wloglemma2} does hold for the empty word with $\ell=0$ letters. 
\end{remark}

We are now in a position to derive an explicit expression for $|I(w,m)|$.

\begin{lemma}
\label{count-II}
Let $w \in \{0,1\}^{\ell}$ be reduced, and let $m \geq 0$. Then
$$ |I(w,m)|\;=\;\tfrac{m^2+(\ell+5)m+2}{2}\binom{3m+\ell}{m}\;-\;\tfrac{m^2+(2\ell+9)m+(\ell^2+7\ell+2)}{2}\binom{3m+\ell}{<m}\;. $$
\end{lemma}

\begin{proof}
For $w' \in I(w,m)$, let $w'' = p^i w s^j$ as in Lemma~\ref{wloglemma2}. Thus $I'(w'',m-e)$ is the set of all the words that are obtained from $w$ through $w''$. By the uniqueness of $w''$, the sum of $|I'(w'',m-e)|$ over such $w''$ yields $|I(w,m)|$.

We determine how many such intermediate words $w''$ exist for each $e$. If $e=0$ then $w''=w$, giving just one option. Otherwise there are $4e$ possibilites, as follows: four options for choosing $p$ and~$s$ for each of the $e-1$ partitions $e=i+j$ where $i,j>0$, and two options for each of the extremal cases $i=0$ or $j=0$. Summing the expression from Proposition~\ref{count-I} over all $w''$ yields 
\begin{align*}
\left|I(w,m)\right| \;&=\; \binom{3m+\ell}{m} - \binom{3m+\ell}{<m} + \sum\limits_{e=1}^m 4e\left[\binom{3m+\ell}{m-e} - \binom{3m+\ell}{<m-e}\right] \\
\;&=\; \binom{3m+\ell}{m} -
\sum\limits_{e=1}^m \left(2e^2-6e+1\right) \binom{3m+\ell}{m-e}.
\end{align*}
With $n=3m+\ell$, we change the variable in the summation to $k=m-e$ and then simplify it via the following two identities:
\begin{align*}
&\sum\limits_{k=0}^{m-1}k\binom{n}{k} \;=\; \frac{n}{2} \binom{n}{<m} \;-\; \frac{m}{2} \binom{n}{m} \\
&\sum\limits_{k=0}^{m-1}k(k-1)\binom{n}{k} \;=\; \frac{n(n-1)}{4} \binom{n}{<m} \;-\; \frac{m(2m+n-3)}{4} \binom{n}{m},
\end{align*}
which are easily verified by induction.
\end{proof}

\section{The knot probability function}
\label{sec:mainthm}

At this stage, we can derive precise and asymptotic expressions for $P\left[K_n  = K\right]$, the probability that a given knot $K$ of bridge number at most two occurs in the random model. The new closed form expression replaces the recursive solution derived in \cite{CoKr}.

\begin{theorem}
\label{thm:main}
For every two-bridge knot $K$ and $i \in \{0,1\}$ there exist two nonnegative integers, $r = r(K) \in \{1,2,4\}$ and $\ell = \ell_i(K) \equiv i$ (mod $3$), such that the probability of $K$ is
$$ P\left[K_n  = K\right] \;=\; \frac{r}{2^n} \cdot F\left(\frac{n-\ell}{3},\;\ell \right), \;\;\;\;\;\;\;\;\;\;\; n \in \{\ell, \ell+3, \ell+6, \dots \}, $$
where
$$ F(m,\ell) \;:=\; \tfrac{m^2+(\ell+5)m+2}{2} \binom{3m+\ell}{m} - \tfrac{m^2 + (2\ell+9)m + (\ell^2+7\ell+2)}{2} \binom{3m+\ell}{<m}. $$

The probability that $K_n$ is the unknot is $ F\left(\left\lfloor n/3 \right\rfloor,0\right) / 2^{3\left\lfloor n/3 \right\rfloor}$.
\end{theorem}

\begin{proof}
As explained in Section~\ref{sec:notation}, each two-bridge knot $K$ can be represented by reduced words of two possible lengths, $\ell_0(K)$ and $\ell_1(K)$, such that $\ell_i(K) \equiv i$~mod~$3$. Also recall that the number of reduced representations of the knot $K$ of length $\ell = \ell_i(K)$ is some integer $r = r(K) \in \{1,2,4\}$.

If $w \in \{0,1\}^{\ell}$ is a reduced word representation of the knot $K$, then there are $|I(w,m)|$ words of length $n = \ell + 3m$ that reduce to it. Since these are the reduced representations, any word representing a particular knot $K$ reduces to exactly one of them. Thus the total number of words representing the knot $K$ equals $r|I(w,m)|$. Dividing this by $2^n$, the number of all possible binary words of length $n$, we obtain the desired probability, since all the $2^n$ words were assumed to be equally likely. Substituting $|I(w,m)| = F(m,\ell)$ from Lemma~\ref{count-II} yields the expression stated in the theorem.

We are left with the case of the unknot, whose reduced lengths are $\ell_0=0$ and $\ell_1=1$. In view of Remark~\ref{smalll}, if $n \equiv 0$~mod~$3$ then the probability of the unknot is given by the same expression just derived above with $r=1$ and~$w$ the empty word. The case $n\equiv 1$~mod~$3$ can be derived by some modifications to Lemma~\ref{count-II}.  Specifically, $4e$ should be replaced with $2e+1$ due to nonuniqueness as mentioned in Remark \ref{smalll} following Lemma~\ref{wloglemma2}. One can verify via Pascal's rule that the new expression with $\ell=1$ is equal to the original one with $\ell=0$.
\end{proof}

We now turn to the asymptotic rate at which the knot probabilities decay to zero as $n \to \infty$.

\begin{corollary}
\label{prop:main}
For every two-bridge knot $K$,
$$ P\left[K_n  = K\right] \;\;=\;\; \alpha ^ {\displaystyle(1 + o(1))n}, $$
where $\alpha = \sqrt[3]{\frac{27}{32}} \approx 0.945$.
\end{corollary}

\begin{proof}
As usual, we denote $n = 3m+\ell$, where $\ell$ is a reduced length for the knot $K$. We estimate this probability by examining the counting argument again. 

A rough estimate for (\ref{eq:needlater}) in the proof of Proposition~\ref{count-I} is $\tfrac{1}{n}\tbinom{n}{|L|} \leq \tfrac{n-3|L|+1}{n-|L|+1}\tbinom{n}{|L|} \leq \tbinom{n}{|L|}$, which yields \mbox{$\tfrac{1}{n}\tbinom{n}{m} < |I'(w,m)| < n\tbinom{n}{m}$} after the summation over $|L|$. In Lemma~\ref{count-II}, this count goes through another summation that brings the final count to the range $\tfrac{1}{n}\tbinom{n}{m} < |I(w,m)| < n^3\tbinom{n}{m}$. 

Since polynomial factors will be absorbed into the $o(1)$ term in the exponent, we only need to analyze the asymptotics of $\tbinom{n}{m} = \tbinom{3m+\ell}{m}$.

A standard application of Stirling's formula to the binomial coefficient yields that $\tfrac{1}{n}\log\tbinom{n}{pn}$ converges uniformly in $p$ to the binary entropy function $H(p) = -p\log (p)-(1-p)\log(1-p)$, where it is customary to use a base-two logarithm. 

Since $\ell$ is constant, this means that $\frac{1}{3m+\ell}\log\tbinom{3m+\ell}{m} \to H(\tfrac13)$ as $m \to \infty$. By the above discussion, it follows that 
$$ \tfrac{1}{n}\log P\left[K_n  = K\right] \;\to\; \left(H(\tfrac13)-1\right) \;=\; \tfrac13 \log\tfrac{27}{32} \;=\; \log\alpha, $$ 
as $n \to \infty\;$.
\end{proof}

\section{Crossing number distribution}
\label{sec:crossing}

With Theorem \ref{thm:main} in hand, we proceed to a study of the probability mass function of the crossing number of a random two-bridge knot~$K_n $ and find its asymptotics for large $n$.

\begin{theorem}
\label{cor:main}
For $n \equiv 0$ or $1$ mod $3$ and $c \in \{3,\dots,n\}$,
$$ P\left[\text{\textnormal{c}}\left(K_n \right) = c\right] \;=\;   \sum\limits_{\substack{k \in \{0,\dots,c-2\} \\ c+k \equiv n \text{\textnormal{ (mod 3)}}}} \frac{\binom{c-2}{k}}{2^{n-1}} \cdot F\left(\frac{n-c-k}{3}, \;c+k\right) $$
where $F(m,\ell)$ is as in Theorem~\ref{thm:main}.
\end{theorem}

\begin{proof}[Proof of Theorem~\ref{cor:main}]
By Proposition~\ref{cross}, the crossing number of a two-bridge knot is equal to the number of runs in the corresponding reduced word. We consider reduced words with $c$ runs, and let $k$ be the number of two-letter runs so that $k \in \{0,\dots,c-2\}$. Choosing which of the $c-2$ internal runs would have two letters and whether the first run would be a $0$ or a $1$ determines a reduced word, and the number of such words is~$2\tbinom{c-2}{k}$.

By Lemma~\ref{count-II}, a reduced word of size $\ell$ can be obtained by $m$ reduction moves from $F(m,\ell)$ different words of $n=\ell+3m$ letters. Hence the number of $n$-letter words that reduce to a given word of $\ell=c+k$ letters is $F(\tfrac{n-c-k}{3},c+k)$ if $n \equiv c+k$ mod $3$, and $0$ otherwise.

The formula is obtained by combining the two above numbers for every admissible value of $k$ and dividing by the total number of $n$-letter words, which is $2^n$.
\end{proof}

\begin{corollary}\label{cor_beta}
$$ \text{\textnormal{c}}\left(K_n \right) \;\;=\;\; (\beta + o(1))n $$
in probability, where $\beta = \frac{\sqrt{5}-1}{4} \approx 0.309 $.
\end{corollary}

\begin{proof}
We use a rough estimate
$$ F\left(\frac{n-\ell}{3} ,\;\ell\right) \;\leq\; \frac{n^2}{2} \binom{n}{(n-\ell)/3}, $$
which yields
$$ P\left[\text{\textnormal{c}}\left(K_n \right) = c\right] \;\leq\; \sum\limits_k \frac{n^2}{2^n} \binom{c}{k} \binom{n}{(n-c-k)/3} \;\leq\; \;\max\limits_{k\in\{0,\dots,c\}} \;\frac{n^3}{2^n} \binom{c}{k} \binom{n}{(n-c-k)/3}.$$
Let $\varepsilon>0$. We will show that this bound is $o(1/n)$ uniformly for all $c \not\in [(\beta-\varepsilon)n,(\beta+\varepsilon)n]$. Then it would follow that the probability of such a crossing number is $o(1)$ for every $\varepsilon$ giving $\text{\textnormal{c}} \left(K_n \right)/n \to \beta$ in probability, as required. 

Indeed, estimating the binomial coefficients by the binary entropy function as in the proof of Corollary \ref{prop:main}, we see that
$$ \frac{1}{1/n}\cdot\frac{n^3}{2^n} \binom{c}{k} \binom{n}{(n-c-k)/3} \;=\; 2^{{\displaystyle -n+H\left(\frac{k}{c}\right)c+H\left(\frac{n-c-k}{3n}\right)n+O(\log n)}}. $$
Thus we need to show
$$ 2^{{\displaystyle \left(\phi(x,y)+o(1)\right)n}} \;=\; o(1), $$
where $x=c/n \neq \beta$ and $y=k/n$, and
$$ \phi(x,y) \;=\; H\left(\frac{y}{x}\right) \cdot x + H\left(\frac{1-x-y}{3}\right) - 1. $$
We verify that $\phi(x,y) < 0$ for $x \neq \beta$. Indeed,
$$ \vec\nabla\phi \;=\; \left( \substack{\log(1-x-y)/3-\log(2+x+y)/3-\log(1-y/x) \\ \log(1-x-y)/3-\log(2+x+y)/3+\log(x/y-1) }\right). $$
Solving for a critical point $\phi_x=\phi_y=0$ yields a unique solution $x_0=(\sqrt{5}-1)/4 = \beta$ and $y_0=(\sqrt{5}-2)/2$, at which $\phi(x_0,y_0)=0$. This is a local maximum, as can be verified by checking the positivity of the discriminant $\phi_{xx} \phi_{yy} - \phi_{xy}^2$, while $\phi_{xx}$ is negative at that point. However, it is also a global maximum because $\phi$ is easily shown to be strictly negative on the boundaries $y=0$, $y=x$, and $x+y=1$.
\end{proof}

\bibliographystyle{amsalpha}
\bibliography{16FebBibliography}

\end{document}